\documentclass{amsart}[12pt]

\let\oldmarginpar\marginpar
\renewcommand\marginpar[1]{\-\oldmarginpar[\raggedleft\footnotesize #1]%
{\raggedright\footnotesize #1}}

\usepackage{xypic}
\usepackage{eucal}

\usepackage{amsmath}
\usepackage{amstext}
\usepackage{amssymb}
\usepackage{amsthm}
\usepackage{amscd}

\usepackage{lscape}
\usepackage{longtable}

\usepackage{epsfig}
\input txdtools
\let\et=\etexdraw
\def\etexdraw{\drawbb\et}

\theoremstyle{plain}
\newtheorem{thm}{Theorem}[section]
\newtheorem{thm*}{Theorem}
\newtheorem{lem}[thm]{Lemma}
\newtheorem{prop}[thm]{Proposition}
\newtheorem{prop*}[thm*]{Proposition}
\newtheorem{cor}[thm]{Corollary}

\theoremstyle{definition}

\newtheorem{ex}[thm]{Example}

\newtheorem{notation}[thm]{Notation}

\theoremstyle{remark}

\DeclareMathOperator{\Ker}{Ker}
\DeclareMathOperator{\Coker}{Coker}
\DeclareMathOperator{\Image}{Im}

\DeclareMathOperator{\Hom}{Hom}

\DeclareMathOperator{\Ann}{ann}
\DeclareMathOperator{\Nil}{Nil}

\DeclareMathOperator{\Matrix}{\mathbf{M}}
\newcommand{\twisted}[1]{F_* #1}

\begin{document}

\title
[Properties and applications of $F$-finite $F$-modules]
{Some properties and applications of $F$-finite $F$-modules}

\author{Mordechai Katzman}
\address{Department of Pure Mathematics,
University of Sheffield, Hicks Building, Sheffield S3 7RH, United Kingdom}
\email{M.Katzman@sheffield.ac.uk}
\thanks{The author gratefully acknowledges support from EPSRC grant EP/G060967/1.}




\begin{abstract}
M.~Hochster's work in  \cite{Hochster}
has shown that $F$-finite $F$-modules over regular local rings 
have finitely many $F$-submodules.
In this paper we apply this theorem to prove that morphisms of $F$-finite $F$-modules have a particularly simple form and
we also show that there exist finitely many submodules compatible with a given Frobenius near-splitting
thus generalizing a similar result in \cite{Blickle-Bockle} to the case where the base ring is not $F$-finite.
\end{abstract}

\maketitle

\section{Introduction}\label{Section: Introduction}

The purpose of this paper is to describe several applications of finiteness properties of $F$-finite $F$-modules recently discovered by M.~Hochster in
\cite{Hochster}
to the study of Frobenius maps on injective hulls, Frobenius near-splittings and to the nature of morphisms of $F$-finite $F$-modules.

Throughout this paper  $(R, m)$ shall denote a complete regular local ring of prime characteristic $p$.
At the heart of everything in this paper is the Frobenius map
$f: R \rightarrow R$ given by $f(r)=r^p$ for  $r\in R$.
We can use this Frobenius map to define a new $R$-module structure on $R$ given by
$r \cdot s=r^p s$; we denote this $R$-module $\twisted{R}$. We can then use this to define the
\emph{Frobenius functor} from the category of $R$-modules to itself: given an $R$-module $M$ we define
$F_R(M)$ to be $\twisted{R} \otimes_R M$ with $R$-module structure given by
$r (s\otimes m)= rs \otimes m$ for $r,s\in R$ and $m\in M$.
Henceforth we shall abbreviate $F_R$ to $F$ for the sake of readability.

Let $R[\Theta; f]$ be the skew polynomial ring
which is the free $R$-module $\displaystyle \oplus_{i=0}^\infty R \Theta^{i}$
with multiplication  $\Theta r=r^{p}\Theta$ for all $r\in R$.
As in \cite{Katzman1}, $\mathcal{C}$ shall denote the category of $R[\Theta ;f]$-modules which are Artinian as $R$-modules.
For any two such modules $M, N$, we denote the morphisms between them in $\mathcal{C}$ with $\Hom_{R[\Theta; f]}(M,N)$;
thus an element $g\in \Hom_{R[\Theta; f]}(M,N)$ is an $R$-linear map such that $g( \Theta a)=\Theta g(a)$ for all $a\in M$.
The first main result of this paper (Theorem \ref{Theorem: injectivity of Lyubeznik's functor}) shows that under some conditions
on $N$, $\Hom_{R[\Theta; f]}(M,N)$ is a finite set.

An \emph{$F$-module}
(cf.~the seminal paper \cite{Lyubeznik} for an introduction to $F$-modules and their properties)
over the ring $R$ is  an $R$-module $\mathcal{M}$ together with an $R$-module isomorphism
$\theta_\mathcal{M} : \mathcal{M} \rightarrow F(\mathcal{M})$.
This isomorphism $\theta_\mathcal{M}$ is the \emph{structure morphism} of $\mathcal{M}$.

A \emph{morphism of $F$-modules} $\mathcal{M} \rightarrow \mathcal{N}$ is an $R$-linear map $g$ which makes the following diagram commute
\begin{equation*}\label{CD9}
\xymatrix{
\mathcal{M} \ar@{>}[d]_{\theta_\mathcal{M}} \ar@{>}[r]^{g} & \mathcal{N} \ar@{>}[d]^{\theta_\mathcal{N}}  \\
F(\mathcal{M})  \ar@{>}[r]^{F(g)} & F(\mathcal{N})
}
\end{equation*}
where $\theta_\mathcal{M}$ and $\theta_\mathcal{N}$ are the structure isomorphisms of $\mathcal{M}$ and $\mathcal{N}$, respectively.
We denote $\Hom_{\mathcal{F}}(\mathcal{M}, \mathcal{N})$ the $R$-module of all morphism of $F$-modules $\mathcal{M} \rightarrow \mathcal{N}$

Given any finitely generated $R$-module $M$ and $R$-linear map $\beta : M \rightarrow F(M)$ one can obtain an $R$-module
$$\mathcal{M}=\underrightarrow{\lim} \left( M \xrightarrow[]{\beta} F(M) \xrightarrow[]{F(\beta)} F^2(M) \xrightarrow[]{F^2(\beta)} \dots \right) .$$
Since
$$F(\mathcal{M})=\underrightarrow{\lim} \left(F(M) \xrightarrow[]{F(\beta)} F^2(M) \xrightarrow[]{F^2(\beta)} F^3(M) \xrightarrow[]{F^3(\beta)} \dots \right)=\mathcal{M}$$
we obtain an isomorphism $\mathcal{M} \cong F(\mathcal{M})$, and hence $\mathcal{M}$ is an $F$-module.
Any $F$-module which can be constructed as a direct limit as $\mathcal{M}$ above is called an \emph{$F$-finite} $F$-module with \emph{generating morphism} $\beta$.

There is a close connection between $R[\Theta; f]$-modules and $F$-finite $F$-modules given by \emph{Lyubeznik's Functor} from $\mathcal{C}$
to the category of $F$-finite $F$-modules which is defined as follows (see section 4 in \cite{Lyubeznik} for the details of the construction.)
Given an $R[\Theta; f]$-module $M$ one defines the $R$-linear map $\alpha : F(M) \rightarrow M$ by $\alpha( r \otimes m)=r \Theta m$;
an application of Matlis duality then yields an $R$-linear map $\alpha^\vee : M^\vee \rightarrow F(M)^\vee \cong F(M^\vee)$
and one defines
$$\mathcal{H} (M)=
\underrightarrow{\lim} \left( M^\vee \xrightarrow[]{\alpha^\vee} F(M^\vee) \xrightarrow[]{F(\alpha^\vee)} F^2(M^\vee) \xrightarrow[]{F^2(\alpha^\vee)} \dots \right) .$$
Since $M$ is an Artinian $R$-module, $M^\vee$ is finitely generated and $\mathcal{H} (M)$ is an $F$-finite $F$-module
with generating morphism $M^\vee \xrightarrow[]{\alpha^\vee} F(M^\vee)$. This construction is functorial and results in an exact
covariant functor from $\mathcal{C}$ to the category of $F$-finite $F$-modules.

Later in this paper we will need the following related constructions.
Following \cite{Katzman1} we shall denote
$\mathcal{D}$ the category of all $R$-linear maps $M \rightarrow F(M)$ where $M$ is any
finitely generated $R$-module,  and where a morphism between $M\xrightarrow[]{a} F(M)$ and
$N\xrightarrow[]{b} F(N)$ is a commutative diagram of $R$-linear maps
\begin{equation*}
\xymatrix{
M \ar@{>}[d]^{a} \ar@{>}[r]^{\mu} & N \ar@{>}[d]^{b}\\
F(M) \ar@{>}[r]^{F(\mu)} & F(N)\\
}
\end{equation*}
Section 3 of \cite{Katzman1}  constructs a pair of functors
$\Delta: \mathcal{C} \rightarrow \mathcal{D}$ and
$\Psi: \mathcal{D} \rightarrow \mathcal{C}$
with the property that for all $L\in \mathcal{C}$,
the $R[\Theta; f]$-module $\Psi \circ \Delta (L)$ is canonically isomorphic to $L$
and for all $D=(B\xrightarrow[]{u} F(B))\in \mathcal{D}$,
$\Delta \circ \Psi (D)$ is canonically isomorphic to $D$.
The functor $\Delta$ amounts to the ``first step'' in the construction of Lyubeznik's functor $\mathcal{H}$:
for $L\in\mathcal{C}$ we define the $R$-linear map $\alpha: F(L) \rightarrow L$ to be the one given above and we
let $\Delta(L)$ to be the map $\alpha^\vee: L^\vee \rightarrow F(L)^\vee\cong F(L^\vee)$
(cf.~section 3 in \cite{Katzman1} for the details of the construction.)

The main result in \cite{Hochster} is the surprising fact that for $F$-finite $F$-modules $\mathcal{M}$ and $\mathcal{N}$,
$\Hom_{\mathcal{F}}(\mathcal{N}, \mathcal{M})$ is a finite set. In section \ref{Section: Morphisms in C} of this
paper we exploit this fact to prove the second main result in this paper (Theorem \ref{Theorem: surjectivity of Lyubeznik's functor})
to show the following.
Let
$\gamma: M \rightarrow F(M)$ and  $\beta: N \rightarrow F(N)$ be generating morphisms for
$\mathcal{M}$ and $\mathcal{N}$.
Given an $R$-linear map $g$ which makes the following diagram commute,
\begin{equation*}
\xymatrix{
N \ar@{>}[d]^{g} \ar@{>}[r]^{\beta}   & F(N) \ar@{>}[d]^{F(g)}  \\
M \ar@{>}[r]^{\gamma}                    & F(M)
}
\end{equation*}
one can extend that diagram to
\begin{equation*}
\xymatrix{
N \ar@{>}[d]^{g} \ar@{>}[r]^{\beta}   & F(N) \ar@{>}[d]^{F(g)} \ar@{>}[r]^{F(\beta)} &  F^2 (N) \ar@{>}[d]^{F^2(g)}  \ar@{>}[r]^{F^2(\beta)} & \dots\\
M \ar@{>}[r]^{\gamma}                 & F(M) \ar@{>}[r]^{F(\gamma)} & F^2(M) \ar@{>}[r]^{F^2(\gamma)}& \dots
}
\end{equation*}
and obtain a map between the direct limits of the horizontal sequences, i.e., an element
in $\Hom_{\mathcal{F}} (\mathcal{N}, \mathcal{M})$.
We prove that all elements in $\Hom_{\mathcal{F}} (\mathcal{N}, \mathcal{M})$
arise in this way (cf.~Theorem \ref{Theorem: surjectivity of Lyubeznik's functor}), thus
morphisms of $F$-finite $F$-modules have a particularly simple form.
This answers a question implicit in \cite[Remark 1.10(b)]{Lyubeznik}.

Finally, in section \ref{Section: Applications to Frobenius splittings} we consider
the module $\Hom_{R} (\twisted{R}^n, R^n)$ of  \emph{near-splittings} of $\twisted{R}^n$.
We establish a correspondence between these near-splittings and Frobenius actions on $E^n$ which enables us to prove the third main result in this paper
(Theorem \ref{Corollary: finitely many submodules}) which asserts that given a near-splitting $\phi$ corresponding to an
injective Frobenius action,
there are finitely many $\twisted{R}$-submodules $V\subseteq \twisted{R}^n$ such that $\phi(V)\subseteq V$.
This generalizes a similar result in \cite{Blickle-Bockle} to the case where $R$ is not $F$-finite.

Our study of Frobenius near-splittings is based on the study of its dual notion, i.e., Frobenius maps on
the injective
hull $E=E_R(R/m)$ of the residue field of $R$.
This injective hull is given explicitly as the module of inverse polynomials $\mathbb{K}[\![ x_1^-, \dots, x_d^- ]\!]$
where $x_1, \dots, x_d$ are minimal generators of the maximal ideal of $R$ (cf.~\cite[\S 12.4]{Brodmann-Sharp}.)
Thus $E$ has a natural $R[T; f]$-module structure
extending $T \lambda x_1^{-\alpha_1} \dots x_1^{-\alpha_d}= \lambda^p x_1^{-p \alpha_1} \dots x_d^{-p \alpha_d}$ for
$\lambda\in \mathbb{K}$ and $\alpha_1, \dots, \alpha_d>0$.
We can further extend this to a natural $R[T; f]$-module structure on $E^n$ given by
$$T \left( \begin{array}{c} a_1 \\ \vdots \\ a_n  \end{array} \right) =  \left( \begin{array}{c} T a_1 \\ \vdots \\ T a_n \end{array} \right) .$$
Throughout this paper $T$ will denote this natural Frobenius map, while $\Theta$ will be uses for general Frobenius maps.

The results of section \ref{Section: Applications to Frobenius splittings} will follow from the fact that there is a dual correspondence between
Frobenius near-splittings and sets of $R[\Theta; f]$-module structures on $E^n$.

\section{Frobenius maps of Artinian modules and their stable submodules}
\label{Section: Frobenius maps of $E^n$ and their stable submodules}

Given an Artinian $R$-module $M$ we can embed $M$ in $E^\alpha$ for some $\alpha\geq 0$ and extend this inclusion to an exact sequence
$$ 0\rightarrow M\rightarrow E^\alpha \xrightarrow[]{A^t} E^\beta \rightarrow \dots$$
where
$A^t\in \Hom_R(E_R^\alpha, E_R^\beta)$.
In our setup Matlis duality gives $\Hom_R(E_R, E_R)\cong R$ and so  $A^t\in \Hom_R(E_R^\alpha, E_R^\beta) \cong \Hom_R(R^\alpha, R^\beta)$ is
a $\beta\times \alpha$ matrix with entries in $R$.
Henceforth in this section we will describe certain properties of Artinian $R$-modules
in terms of their representations as kernels of
matrices with entries in $R$. We shall denote $\Matrix_{\alpha, \beta}$ the set of $\alpha \times \beta$ matrices with entries in $R$
and for any such matrix $A$ we will write $A^{[p]}$ to denote the matrix obtained by raising each of its entries to the $p$th power.

We now explore the duality between $E^\alpha$ with a given $R[\Theta; f]$-module structure and
$R$-linear maps $R^\alpha \rightarrow R^\alpha$ for $\alpha\geq 1$ given by the functors $\Delta$ and $\Psi$ defined in section \ref{Section: Introduction}.
Under this duality the $R[\Theta; f]$-module structure
corresponding to the map $(R^\alpha \rightarrow R^\alpha)\in \mathcal{D}$ given by multiplication by $B\in \Matrix_{\alpha, \alpha}$ is given by
$\Theta=B^t T$ where $T$ is the natural Frobenius map on $E^\alpha$ described in section \ref{Section: Introduction}.

\begin{prop}\label{Proposition: correspondence between Frobenius maps and matrices}
Let $M=\ker A^t\subseteq E^\alpha$ be an  Artinian  $R$-module where $A\in \Matrix_{\alpha, \beta}$.
Let $\mathbf{B} =\left\{ B\in \Matrix_{\alpha, \alpha} \,|\, \Image B A \subseteq \Image A^{[p]} \right\}$.
For any $R[\Theta; f]$-module structure on $M$,
$\Delta(M)$ can be identified with an element in  $\Hom_R(\Coker A, \Coker A^{[p]})$
and thus represented by multiplication by some  $B\in\mathbf{B}$. Conversely,
any such $B$ defines an $R[\Theta; f]$-module structure on $M$ which is given by the restriction to $M$ of the Frobenius map
$\phi: E^\alpha\rightarrow E^\alpha$
defined by $\phi(v)=B^t T(v)$ where $T$ is the natural Frobenius map on $E^\alpha$.
\end{prop}
\begin{proof}
Matlis duality gives an exact sequence
$ R^\beta \xrightarrow[]{A} R^\alpha \rightarrow M^\vee \rightarrow 0$ hence
$$\Delta(M)\in \Hom_R( M^\vee, F_R(M^\vee))\cong  \Hom_R(\Coker A, \Coker A^{[p]}).$$

Let $\Delta(M)$ be the map $g:\Coker A \rightarrow \Coker A^{[p]}$.

In view of Theorem 3.1 in \cite{Katzman1} we only need to show that
any such $R$-linear map is given by multiplication by an $B\in \mathbf{B}$, and that any such $B$
defines an element in $\Delta(M)$.

Using the freeness of $R^\alpha$, we find a map $g^\prime$ which makes the following  diagram
\begin{equation*}\label{CD1}
\xymatrix{
R^\alpha \ar@{>}[rd]^{g^\prime} \ar@{>}[r]^{q_1}   & R^\alpha/\Image A \ar@{>}[r]^{g}    & R^\alpha/\Image A^{[p]}\\
                                            & R^\alpha \ar@{>}[ur]^{q_2}                                    & \\
}
\end{equation*}
commute, where $q_1$ and $q_2$ are quotient maps.
The map $g^\prime$ is given by multiplication by some $\alpha \times \alpha$ matrix $B\in \mathbf{B}$.
Conversely, any such matrix $B$ defines a map $g$ making the diagram above commute, and $\Psi(g)$ gives a
$R[\Theta; f]$-module structure on $M$ as described in the last part of the proposition.
\end{proof}

\begin{notation}
We shall henceforth describe Artinian $R$-modules with a given $R[\Theta; f]$-module structure in terms of the two matrices
in the statement of Proposition \ref{Proposition: correspondence between Frobenius maps and matrices} and talk about
Artinian $R$-modules $M=\Ker A^t\subseteq E^\alpha$ where $A\in \Matrix_{\alpha,  \beta}$ with $R[\Theta; f]$-module structure
given by $B\in \Matrix_{\alpha,  \alpha}$.
\end{notation}

\section{Morphisms in $\mathcal{C}$}\label{Section: Morphisms in C}

In this section we raise two questions.
The first of these asks when for given $R[\Theta; f]$-modules   $M, N$,  the set $\Hom_{R[\Theta; f]} (M,N)$ is finite;
later in this section we prove that this holds when $N$ has no $\Theta$-torsion.
The following two examples illustrate why this set is not finite in general,
and why it is finite in a special simple case.

\begin{ex}
Let $\mathbb{K}$ be an infinite field of prime characteristic $p$ and let $R=\mathbb{K}[\![ x ]\!]$.
Let $M=\Ann_E xR$ and fix an $R[\Theta; f]$-module structure on $M$ given by $\Theta a=x^p T a$ where $T$ is the standard Frobenius action on $E$.
Note that $\Theta M=0$ and that for all $\lambda\in \mathbb{K}$ the map $\mu_\lambda : M \rightarrow M$ given by multiplication by $\lambda$ is
in  $\Hom_{R[\Theta; f]} (M,M)$, and hence this set is infinite.
\end{ex}

\begin{ex}
Let $I,J\subseteq R$ be ideals, and fix  $u\in (I^{[p]} : I)$ and  $v\in (J^{[p]} : J)$.
Endow $\Ann_E I$ and $\Ann_E J$ with $R[\Theta; f]$-module structures given by $\Theta a = uT a$ and $\Theta b = vT b$ for $a\in \Ann_E I$  and
$b\in \Ann_E J$ where $T$ is the standard Frobenius map on $E$.

If $g: \Ann_E I \rightarrow \Ann_E J$ is $R$-linear, an application of Matlis duality yields $g^\vee : R/J \rightarrow R/I$ and we deduce that
$g$ is given by multiplication by an element in $w\in (I : J)$.
If in addition $g\in \Hom_{R[\Theta; f]} (\Ann_E I, \Ann_E J)$, we must have
$ w u T a = g( \Theta a) = \Theta g(a) = v T w a= v w^p T a$,
for all $a\in \Ann_E I$, hence $(v w^p  - u w) T \Ann_E  I=0$ and
$v w^p  - u w \in I^{[p]}$.
The finiteness of $\Hom_{R[\Theta; f]} (\Ann_E I, \Ann_E J)$ translates in this setting to the finiteness of the set of solutions modulo $I^{[p]}$
for the variable $w$
of the equation above, and it is not clear why this set should be finite.
However, if we simplify to the case where $I=J=0$,
the set of solutions of $v w^p - u w=0$ over the the fraction field of $R$ has at most $p$ elements, and in this case we can deduce that
$\Hom_{R[\Theta; f]} (E, E)$ has also at most $p$ elements.
\end{ex}

As in \cite{Lyubeznik}, for any $R[\Theta; f]$-module $M$ we define the \emph{submodule of nilpotent elements}
to be
$\Nil(M)=\{ a\in M \,|\, \Theta^e a = 0 \text{ for some } e\geq 0\}$.
We recall that when $M$ is an Artinian $R$-module there exists an $\eta\geq 0$ such that
$\Theta^\eta \Nil(M)=0$ (cf.~\cite[Proposition 1.11]{Hartshorne-Speiser} and \cite[Proposition 4.4]{Lyubeznik}.)
We also define $M_{\text{red}}=M/\Nil(M)$ and $M^*=\cap_{e\geq 0} R\Theta^e M$ where $R\Theta^e M$ denotes the $R$-module generated by
$\{\Theta^e a \,|\, a\in M\}$. We also note that when $M$ is an $R[\Theta; f]$-module which is  Artinian as an $R$-module, there exists an $e\geq 0$ such that $M^*=R\Theta^e M$ and also
$( M_{\text{red}} )^*=(M^*)_{\text{red}}$ (cf.~section 4 in \cite{Katzman2}.)

\begin{thm}\label{Theorem: injectivity of Lyubeznik's functor}
Let  $M, N$ be $R[\Theta; f]$-modules and let $\phi\in \Hom_{R[\Theta; f]} (M,N)$.
We have $\mathcal{H}(\Image \phi)=0$ if and only if $\phi(M)\subseteq \Nil(N)$ and, consequently,
if $\Nil(N)=0$, the map $\mathcal{H}: \Hom_{R[\Theta; f]} (M,N) \rightarrow \Hom_{\mathcal{F}_R} (\mathcal{H}(N), \mathcal{H}(M))$
is an injection and $\Hom_{R[\Theta; f]} (M,N)$ is a finite set.
\end{thm}
\begin{proof}

We apply $\mathcal{H}$ to the commutative diagram
\begin{equation*}\label{CD7}
\xymatrix{
M \ar@{>>}[d]_{\phi} \ar@{>}[dr]^{\phi} &  \\
\Image \phi \ar@{^{(}->}[r]  & N \\
}
\end{equation*}
to obtain the commutative diagram
\begin{equation*}\label{CD8}
\xymatrix{
\mathcal{H}(N) \ar@{>}[dr]_{\mathcal{H}(\phi)} \ar@{>>}[r] & \mathcal{H}(\Image \phi) \ar@{^{(}->}[d]  \\
& \mathcal{H}(M)  \\
}
.
\end{equation*}
Now $\mathcal{H}(\phi)=0$ if and only if $\mathcal{H}(\Image \phi)=0$, and by \cite[Theorem 4.2]{Lyubeznik} this is equivalent to
$\left(\Image \phi\right)^*_\text{red}=0$.

Choose $\eta\geq 0$ such that $\Theta^\eta \Nil(N)=0$
and choose $e\geq 0$ such that $\left(\Image \phi\right)^*=R\Theta^e \Image \phi$.

Now
\begin{eqnarray*}
\left(\Image \phi\right)^*_\text{red}=0 & \Leftrightarrow & R\Theta^\eta R\Theta^e \phi(M)=0\\
& \Leftrightarrow & R\Theta^{\eta+e}  \phi(M)=0\\
& \Leftrightarrow & \Image \phi \subseteq \Nil(N)
\end{eqnarray*}

The second statement now follows immediately.
\end{proof}

The second main result in this section,  Theorem \ref{Theorem: surjectivity of Lyubeznik's functor}
shows that all morphisms of $F$-finite $F$-modules arise as images of maps of $R[\Theta; f]$-modules under Lyubeznik's functor $\mathcal{H}$.

\begin{thm}\label{Theorem: surjectivity of Lyubeznik's functor}
Let $\mathcal{M}$ and $\mathcal{N}$ be $F$-finite $F$-modules.
For every $\phi \in \Hom_{\mathcal{F}_R} (\mathcal{N}, \mathcal{M})$
there exist generating morphisms $\gamma: M \rightarrow F(M)\in \mathcal{D}$ and  $\beta: N \rightarrow F(N) \in \mathcal{D}$
for $\mathcal{M}$ and $\mathcal{N}$, respectively, and a morphism (in the category $\mathcal{D}$)
\begin{equation*}\label{CD10}
\xymatrix{
N \ar@{>}[d]^{g} \ar@{>}[r]^{\beta}   & F(N) \ar@{>}[d]^{F(g)}  \\
M \ar@{>}[r]^{\gamma}                    & F(M)
}
\end{equation*}
such that $\phi=\mathcal{H} \left(\Psi(g)\right)$, i.e., such that
$\phi$ is the map of direct limits
\begin{equation*}\label{CD10a}
\xymatrix{
N \ar@{>}[d]^{g} \ar@{>}[r]^{\beta}   & F(N) \ar@{>}[d]^{F(g)}   \ar@{>}[r]^{F(\beta)} &  F^2(N) \ar@{>}[d]^{F^2(g)}   \ar@{>}[r]^{F^2(\beta)} & \dots \\
M \ar@{>}[r]^{\gamma}                 & F(M) \ar@{>}[r]^{F(\gamma)}                    &  F^2(M) \ar@{>}[r]^{F^2(\gamma)}                      & \dots
}
\end{equation*}


\end{thm}
\begin{proof}

Choose any generating morphisms
$$\mathcal{N}=\underrightarrow{\lim} \left( N \xrightarrow[]{\beta} F(N) \xrightarrow[]{F(\beta)} F^2(N) \xrightarrow[]{F^2(\beta)} \dots \right)$$
and
$$\mathcal{M}=\underrightarrow{\lim} \left( M \xrightarrow[]{\gamma} F(M) \xrightarrow[]{F(\gamma)} F^2(M) \xrightarrow[]{F^2(\gamma)} \dots \right)$$
and fix any $\phi \in \Hom_{\mathcal{F}_R} (\mathcal{N}, \mathcal{M})$.

For all $j\geq 0$ let $\phi_j$ be the restriction of $\phi$ to the image of $F^j(N)$ in $\mathcal{N}$.

The fact that $\phi$ is a morphism of $F$-modules implies that for every $j\geq 0$ we have a commutative diagram
\begin{equation*}\label{CD4}
\xymatrix{
F^{j}(N) \ar@{>}[d]^{} \ar@{>}[r]^{F^{j}(\beta)}   & F^{j+1}(N) \ar@{>}[d]^{}  \\
\mathcal{N} \ar@{>}[d]^{\phi} \ar@{>}[r]^{\theta_\mathcal{N}}   & F(\mathcal{N}) \ar@{>}[d]^{F(\phi)}  \\
\mathcal{M}   \ar@{>}[r]_{\theta_\mathcal{M}}^{\cong}                    & F(\mathcal{M})   \\
}
\end{equation*}
where $\theta_\mathcal{M}$ and $\theta_\mathcal{N}$ are the structure isomorphims of $\mathcal{M}$ and $\mathcal{N}$, respectively, and
where the compositions of the vertical maps are $\phi_j$ and $F(\phi_j)$.
Repeated applications of the Frobenius functor yields a commutative diagram
\begin{equation*}\label{CD5}
\xymatrix{
F^j(N) \ar@{>}[d]^{\phi_j} \ar@{>}[r]^{F^j(\beta)}    & F^{j+1}(N) \ar@{>}[d]^{F(\phi_j)} \ar@{>}[r]^{F^{j+1}(\beta)} & \dots \\
\mathcal{M}   \ar@{>}[r]_{\theta_\mathcal{M}}^{\cong} &  F(\mathcal{M})   \ar@{>}[r]_{F(\theta_\mathcal{M})}^{\cong}  & \dots \\
}
\end{equation*}
and we can now extend this commutative diagram to the left to obtain
\begin{equation*}\label{CD6}
\xymatrix{
N \ar@{>}[ddrrr]^{\phi_0} \ar@{>}[r]^{\beta} & F(N) \ar@{>}[ddrr]^{\phi_1} \ar@{>}[r]^{F(\beta)}& \dots \ar@{>}[r]^{F^{j-1}(\beta)} &  F^{j}(N) \ar@{>}[dd]^{\phi_{j}} \ar@{>}[r]^{F^{j}(\beta)} & F^{j+1}(N) \ar@{>}[d]^{F(\phi_j)} \ar@{>}[r]^{F^{j+1}(\beta)}  & F^{j+2}(N) \ar@{>}[d]^{F^{2}(\phi_j)} \ar@{>}[r]^{F^{j+2}(\beta)} & \dots \\
&&&& F(\mathcal{M}) \ar@{>}[dl]_{\theta_\mathcal{M}^{-1}}  & F^2(\mathcal{M}) \ar@{>}[dll]^{\theta_\mathcal{M}^{-1}\circ F(\theta_\mathcal{M})^{-1}} \quad\dots \\
&&&\mathcal{M} &&\\
}
\end{equation*}
This commutative diagram defines an $R$-linear map $\psi_j: \mathcal{N} \rightarrow \mathcal{M}$.
Furthermore, we show next that this $\psi_j$ is a map of $\mathcal{F}$-modules, i.e.,
that for all $j\geq 0$, $F(\psi_j) \circ \theta_\mathcal{N}= \theta_\mathcal{M} \circ \psi_j$.
Fix $j\geq 0$ and abbreviate $\psi=\psi_j$.

Pick any $a\in \mathcal{N}$ represented as an element of $F^e(N)$.
If $e<j$ then the fact that  $\phi$ is a morphism of $F$-modules implies that
$$\theta_\mathcal{M}\circ \psi(a) = \theta_\mathcal{M} \circ \phi(a) = F(\phi) \circ \theta_\mathcal{N} (a) = F(\psi) \circ \theta_\mathcal{N} (a).$$
Assume now that $e\geq j$; we have \begin{eqnarray*}
\theta_\mathcal{M} \circ \psi (a) & = & \theta_\mathcal{M} \circ \theta_\mathcal{M}^{-1} \circ F(\theta_\mathcal{M}^{-1}) \circ \dots
\circ F^{e-1-j}(\theta_\mathcal{M}^{-1}) \circ F^{e-j}(\phi_j) (a) \\
&=&   F(\theta_\mathcal{M}^{-1}) \circ \dots \circ F^{e-1-j}(\theta_\mathcal{M}^{-1}) \circ F^{e-j}(\phi_j) (a)
\end{eqnarray*}

and
\begin{eqnarray*}
F(\psi) \circ \theta_\mathcal{N}(a) & = &
F\left( \theta_\mathcal{M}^{-1} \circ F(\theta_\mathcal{M}^{-1}) \circ \dots \circ F^{e-1-j}(\theta_\mathcal{M}^{-1}) \circ F^{e-j}(\phi_j) \right)
\left( F^e(\beta) (a) \right) \\
& = & F(\theta_\mathcal{M}^{-1}) \circ  \dots \circ F^{e-1-j}(\theta_\mathcal{M}^{-1})  \circ F^{e-j}(\theta_\mathcal{M}^{-1}) \circ F^{e+1-j}(\phi_j) \left( F^e(\beta) (a) \right) \\
& = & F(\theta_\mathcal{M}^{-1}) \circ  \dots \circ F^{e-1-j}(\theta_\mathcal{M}^{-1})  \circ F^{e-j}(\theta_\mathcal{M}^{-1}) \circ F^{e-j}(\theta_\mathcal{M}) \circ F^{e-j} (\phi_j)(a) \\
& = & F(\theta_\mathcal{M}^{-1}) \circ  \dots \circ F^{e-1-j}(\theta_\mathcal{M}^{-1})   \circ F^{e-j} (\phi_j )(a)
\end{eqnarray*}
where the penultimate inequality follows from the fact that $\phi$ is a morphism of $F$-modules.

Consider now the set $\{ \psi_i \}_{i\geq 0}$; it is a finite set according to Theorem 5.1 in \cite{Hochster}, hence we can
find a sequence $0\leq i_1 < i_2 < \dots$ such that $\psi_{i_1}=\psi_{i_2}=\dots$.
By replacing $\mathcal{N}$ and $\mathcal{M}$ with $F^{i_1} (\mathcal{N})$ and $F^{i_1}(\mathcal{M})$ we may assume that $i_1=0$.

Pick $j\geq 0$ so that $\phi$ maps the image of $N$ in $\mathcal{N}$ into $F^j(M)$. Since $\mathcal{M}\cong F^j(\mathcal{M})$ we may replace $\mathcal{M}$ with $F^j(\mathcal{M})$
and assume that $\phi(\Image N)\subseteq M$ and hence also that
for all $e\geq 0$, $F^e(\phi)$ maps the image of $F^e(N)$ in $\mathcal{N}$ into $F^e(M)$.

Fix now any $e\geq 0$ and pick any $i_k>e$; the fact that $\psi_{0}=\psi_{i_k}$ implies that for all $a\in F^e(N)$,
$F^{e}(\phi_{0})(a)=\psi_{0}(a)=\psi_{i_k}(a)=\phi(a)$ and since this holds for all $e\geq 0$ we deduce that $\phi$ is induced from the commutative
diagram
\begin{equation*}\label{CD3}
\xymatrix{
N \ar@{>}[d]^{\phi_0} \ar@{>}[r]^{\beta}   & F(N) \ar@{>}[d]^{F(\phi_0)} \ar@{>}[r]^{F(\beta)}  & F^2(N) \ar@{>}[d]^{F^2(\phi_0)} \ar@{>}[r]^{F^2(\beta)} & \dots \\
M   \ar@{>}[r]^{\gamma}                    & F(M) \ar@{>}[r]^{F(\gamma)}                        & F^2(M)  \ar@{>}[r]^{F^2(\gamma)}                        & \dots \\
}
\end{equation*}

An application of the functor $\Psi$ to the leftmost square in the commutative diagram above yields a morphism of
$R[\Theta; f]$-modules $g: M \rightarrow N$
and $\phi=\mathcal{H}(g)$.
\end{proof}

\section{Applications to Frobenius splittings}\label{Section: Applications to Frobenius splittings}

For any $R$-module $M$ let $\twisted{M}$ denote the additive Abelian group $M$ with $R$-module structure given by
$r \cdot a = r^p a$ for all $r\in R$ and $a\in M$.
In this section we study the module
$\Hom_{R} (\twisted{R}^n, R^n)$ of  \emph{near-splittings} of $\twisted{R}^n$.
Given such an element $\phi\in \Hom_{R} (\twisted{R}^n, R^n)$ we will describe
the submodules $V\subseteq \twisted{R}^n$ for which $\phi(V)\subseteq V$.
These submodules in the case $n=1$,  known as \emph{$\phi$-compatible ideals}, are of significant importance in algebraic geometry (cf.~\cite{Brion-Kumar} for a study of applications
of Frobenius splittings and their compatible submodules in algebraic geometry.)
We will prove that under some circumstances these form a finite set
and thus generalize a result in \cite{Blickle-Bockle} obtained in the $F$-finite case.

We first exhibit the following easy implication of Matlis duality necessary for the results of this section.
\begin{lem}\label{Lemma: extension of Matlis duality}
For any (not necessarily finitely generated) $R$-module $M$,
$\Hom_R(M,R)\cong \Hom_R(R^\vee, M^\vee)$.
\end{lem}
\begin{proof}
For all $a\in E$ let $h_a\in\Hom_R(R,E)$ denote the map sending $1$ to $a$.

For any $\phi\in \Hom_R(M,R)$, $\phi^\vee\in \Hom_R(R^\vee, M^\vee)$ is defined as
$\left(\phi^\vee(h_a)\right)(m)=\phi(m) a$ for any $m\in M$ and $a\in E$.
For any
$\psi\in \Hom_R(R^\vee, M^\vee)$ we define $\widetilde{\psi}\in \Hom_R(M,R)\cong \Hom_R(M,E^\vee)$
as
$\left(\widetilde{\psi}(m)\right)(a)=\left(\psi(h_a)\right)(m) $ for all $a\in E$ and $m\in M$.
Note that the function $\psi \mapsto \widetilde{\psi}$ is $R$-linear.

Let $\psi\in \Hom_R(R^\vee, M^\vee)$ and fix an $m\in M$.
Note that for all $a\in E$
$$\widetilde{\psi}^\vee(h_a)(m)=\widetilde{\psi}(m) a$$
when we view  $\widetilde{\psi}$ as an element in $\Hom_R(M, R)$.
After we identify $\Hom_R(M, E^\vee)$ with $\Hom_R(M, R)$ we can write
$$\widetilde{\psi}^\vee(h_a)(m)=\widetilde{\psi}(m) (a)= \psi(h_a)(m)$$
thus $\widetilde{\psi}^\vee=\psi$.

It is now enough  to show that for all $\phi\in \Hom_R(M,R)$, $\widetilde{\phi^\vee}=\phi$, and indeed
for all $a\in E$ and $m\in M$
$$\left(\widetilde{\phi^\vee}(m) \right) (a)=\left(\phi^\vee(h_a)\right)(m)=\phi(m) a,$$
i.e., $\left(\widetilde{\phi^\vee}(m)\right)\in \Hom_R(E,E)$ is given by multiplication by $\phi(m)$
and so under the  identification of  $\Hom_R(E,E)$ with $R$, $\widetilde{\phi^\vee}$ is identified with $\phi$.
\end{proof}

We can now prove a generalization Lemma 1.6 in \cite{Fedder} in the form of the next two theorems.

\begin{thm}\label{Theorem: Freeness}
\begin{enumerate}
\item[(a)] The $\twisted{R}$-module $\Hom_R\left(\twisted{R} , E\right)$ is injective of the form
$\oplus_{\gamma\in\Gamma} \twisted{E} \oplus H$ where
$\Gamma$ is non-empty, $H=\bigoplus_{\lambda\in\Lambda} \twisted{E(R/P_\lambda)}$,
$\Lambda$ is a (possibly empty) set, $P_\lambda$ is a non-maximal prime ideal of $R$ for all $\lambda\in\Lambda$ and
$E(R/P_\lambda)$ denotes the injective hull of $R/P_\lambda$.

\item[(b)] Write $\mathcal{B}=  \Hom_{\twisted{R}}\left(E, \oplus_{\gamma\in\Gamma} \twisted{E}\right)\subseteq
\prod_{\gamma\in\Gamma} \Hom_{\twisted{R}}\left(E,  \twisted{E}\right)$.
We have
$$\Hom_{R}\left(\twisted{R}, R \right) \cong \mathcal{B} \subseteq \prod_{\gamma\in\Gamma} \Hom_{\twisted{R}}\left(E,  \twisted{E}\right) \cong
\prod_{\gamma\in\Gamma} \twisted{R} T$$
where $T$ is the standard Frobenius map on $E$.

\item[(c)] The set $\Gamma$ is finite if and only if $\twisted\mathbb{K}$ is a finite extension of $\mathbb{K}$, in which case $\# \Gamma=1$.
\end{enumerate}
\end{thm}
\begin{proof}

The functors $\Hom_{R} \left( - , E\right)=\Hom_{R} \left( - \otimes_{\twisted{R}} \twisted{R}, E\right)$ and
$\Hom_{\twisted{R}} \left( -, \Hom_{R}\left( \twisted{R}, E\right)\right)$ from the category of $\twisted{R}$-modules to itself
are isomorphic by the adjointness of $\Hom$ and $\otimes$, and since $\Hom_{R} \left( - , E\right)$ is an exact functor,
so is $\Hom_{\twisted{R}} \left( -, \Hom_{R}\left( \twisted{R}, E\right)\right)$, thus
$\Hom_{R}\left( \twisted{R}, E\right)$ is an injective $\twisted{R}$-module and hence of the form
$G \oplus H$ where $G$ is a direct sum of copies of $\twisted{E}$ and $H$ is as in the statement of the Theorem.
Write $G=\oplus_{\gamma\in\Gamma} \twisted{E}$. To finish establishing (a) we need only verify that $\Gamma\neq \emptyset$ and we do this below.

Pick any $h\in  \Hom_{R} \left( E, \bigoplus_{\lambda\in\Lambda} \twisted{E(R/P_\lambda)} \right)$.
For any $a\in E$, $h(a)$ can be written as a finite sum $b_{\lambda_1}+ \dots + b_{\lambda_s}$ where
$\lambda_1, \dots, \lambda_s\in\Lambda$ and
$b_{\lambda_1}\in\twisted{E(R/P_{\lambda_1})}, \dots, b_{\lambda_s}\in\twisted{E(R/P_{\lambda_s})}$. Use prime avoidance
to pick a $z\in m \setminus \cup_{i=1}^s P_{\lambda_i}$;
now $z$ and its powers act invertibly on each of $\twisted{E(R/P_{\lambda_1})}, \dots, \twisted{E(R/P_{\lambda_s})}$
while a power of $z$ kills $a$, and so we must have $h(a)=0$. We deduce that
$ \Hom_{R} \left( E,  \bigoplus_{\lambda\in\Lambda} \twisted{E(R/P_\lambda)} \right)=0$ and
\begin{eqnarray*}
 \Hom_{R}\left(E , \Hom_{R}\left( \twisted{R}, E\right)\right)& \cong & \Hom_{R}\left(E, G \oplus \bigoplus_{\lambda\in\Lambda} \twisted{E(R/P_\lambda)} \right)\\
 & \cong & \Hom_{R}\left(E,G \right)  \oplus \Hom_{R} \left( E, \bigoplus_{\lambda\in\Lambda} \twisted{E(R/P_\lambda)} \right) \\
 & \cong &  \Hom_{R}\left(E,G\right) \\
 & \cong & \Hom_{R}\left(E, \oplus_{\gamma\in\Gamma} \twisted{E}\right) \\
 & = & \mathcal{B}.
\end{eqnarray*}
Now $\Hom_{R}\left(E, \twisted{E}\right) $ is the $R$-module of Frobenius maps on $E$ which is isomorphic as an
$\twisted{R}$ module to $\twisted{R} T$ and
we conclude that
$\Hom_{R}\left(E , \Hom_{R}\left( \twisted{R}, E\right)\right) \subseteq \prod_{\gamma\in\Gamma} \twisted{R} T$.

An application of the Matlis dual and Lemma \ref{Lemma: extension of Matlis duality} now gives
$$\Hom_{R}\left(\twisted{R}, R \right)  \cong   \Hom_{R}\left(E , \Hom_{R}\left( \twisted{R}, E\right)\right) $$
and (b) follows.

Write $\mathbb{K}=R/m$ and note that $\twisted{\mathbb{K}}$ is the field extension of $\mathbb{K}$ obtained by adding all $p$th roots of elements in $\mathbb{K}$.
We next compute the cardinality of $\Gamma$ as the $\twisted{\mathbb{K}}$-dimension of $\Hom_{\twisted{\mathbb{K}}}\left( \twisted{\mathbb{K}}, G \right)$.
A similar argument to the one above shows that
$$\Hom_{\twisted{\mathbb{K}}}\left( \twisted{\mathbb{K}},  \bigoplus_{\lambda\in\Lambda} \twisted{E(R/P_\lambda)}  \right) =0$$
hence
$\Hom_{\twisted{\mathbb{K}}}\left( \twisted{\mathbb{K}}, G \right)=\Hom_{\twisted{\mathbb{K}}}\left( \twisted{\mathbb{K}}, \Hom_R\left( \twisted{R}, E\right) \right)$.

We may identify $\Hom_{\twisted{\mathbb{K}}}\left( \twisted{\mathbb{K}}, \Hom_R\left( \twisted{R}, E\right) \right)$ and
$\Hom_{\twisted{R}}\left( \twisted{\mathbb{K}}, \Hom_R\left( \twisted{R}, E\right) \right)$.
Another application of the adjointness of $\Hom$ and $\otimes$ gives
$$\Hom_{\twisted{R}}\left( \twisted{\mathbb{K}}, \Hom_R\left( \twisted{R}, E\right) \right) \cong
\Hom_{R}\left( \twisted{\mathbb{K}} \otimes_{\twisted{R}} \twisted{R}, E \right) \cong \Hom_{R}\left( \twisted{\mathbb{K}}, E \right) .$$

Since $m  \twisted{\mathbb{K}}=0$, we see that the image of any $\phi\in \Hom_{R}\left( \twisted{\mathbb{K}}  , E \right)$
is contained in $\Ann_E m \cong \mathbb{K}$ and we deduce that
$\Hom_{R}\left( \twisted{\mathbb{K}}, E \right)\cong \Hom_{R}\left( \twisted{\mathbb{K}}, \mathbb{K} \right)$.
We can now conclude that the cardinality of $\Gamma$ is the  $\twisted{\mathbb{K}}$-dimension of
$\Hom_{R}\left( \twisted{\mathbb{K}}, \mathbb{K} \right)$. In particular $\Gamma$ cannot be empty and (a) follows.

If $\mathcal{U}$ is a $\mathbb{K}$-basis for $\twisted{\mathbb{K}}$ containing $1\in \twisted{\mathbb{K}}$ ,
\begin{equation}\label{eqn1}
\Hom_{\mathbb{K}}\left( \twisted{\mathbb{K}} , \mathbb{K} \right) \cong
\prod_{b\in \mathcal{U}}  \Hom_{\mathbb{K}}\left( \mathbb{K} b , \mathbb{K} \right)
\end{equation}
and when $\mathcal{U}$ is finite, this is a one-dimensional $\twisted{\mathbb{K}}$-vector space spanned by the projection onto
$\mathbb{K} 1 \subset \twisted{\mathbb{K}}$. If  $\mathcal{U}$ is not finite, the dimension as $\mathbb{K}$-vector space of $(\ref{eqn1})$
is at least $2^{\# \mathcal{U}}$ hence $\Hom_{\mathbb{K}}\left( \twisted{\mathbb{K}} , \mathbb{K} \right)$ cannot be a finite-dimensional
$\twisted{\mathbb{K}}$-vector space.
\end{proof}

Our next result is to establish a connection between submodules of $R^n$ compatible with a given $B \in \Hom_R(F_*R^n, R^n)$
and submodules of $E^n$ fixed under a sequence of Frobenius actions determined by $B$.
Note that the previous theorem allows us to view elements of $\Hom_R(F_*R^n, R^n)\cong \Hom_R(F_*R, R)^{n\times n}=\mathcal{B}^{n\times n}$
as elements in $\prod_{\gamma\in\Gamma} \twisted{R}^{n\times n} T$, i.e.,
as sequences $(B_\gamma T)_{\gamma \in \Gamma}$ where each
$B_\gamma$ is an $n\times n$ matrix with entries in $F_*R$ and $T$ is the natural Frobenius action on $E^n$.

\begin{thm}\label{Theorem: Frobenius and compatability 1}
Let $G=\oplus_{\gamma\in\Gamma} \twisted{E}$ and $\mathcal{B}$ be as in Theorem \ref{Theorem: Freeness}.
Let $B\in \Hom_{R}\left(\twisted{R}^n, R^n \right)$ be represented by $(B_\gamma T)_{\gamma \in \Gamma}\in \mathcal{B}^{n\times n}$.
For all $\gamma\in \Gamma$ consider $E^n$ as an $R[\Theta_\gamma; f]$-module with $\Theta_\gamma v= B_\gamma^t T v$ for all $v\in E^n$.
Let $V$ be an $R$-submodule of $R^n$ and fix a matrix $A$ whose columns generate $V$.
If $B (\twisted{V})\subseteq V$, then
$\Ann_{E^n} A^t$ is a $R[\Theta_\gamma; f]$ submodule of $E^n$ for all $\gamma\in \Gamma$.
\end{thm}
\begin{proof}

Apply the Matlis dual to the commutative diagram
\begin{equation*}\label{CD14}
\xymatrix{
0 \ar@{>}[r]^{}  &
\twisted{V} \ar@{>}[r]^{} \ar@{>}[d]^{B} &
\twisted{R}^n \ar@{>}[r]^{} \ar@{>}[d]^{B}               &
\twisted{R}^n/\twisted{A} \ar@{>}[r]^{}  \ar@{>}[d]^{\overline{B}} &
0  \\
0 \ar@{>}[r]^{} &
V\ar@{>}[r]^{} &
R^n \ar@{>}[r]^{} &
R^n/V \ar@{>}[r]^{} &
0
}
\end{equation*}
where the rightmost vertical map is induced by the middle map to obtain
\begin{equation*}\label{CD15}
\xymatrix{
0 \ar@{>}[r]^{} &
\left(R^n/V\right)^\vee  \ar@{>}[r]^{} \ar@{>}[d]^{\overline{B}^\vee} &
E^n \ar@{>}[d]^{B^\vee} \\
0 \ar@{>}[r]^{}  &
\left(\twisted{R}^n/\twisted{V}\right)^\vee \ar@{>}[r]^{} &
\Hom_R\left( \twisted{R}^n, E \right)
}
\end{equation*}
Note that the previous theorem shows that
$$\Hom_R\left(E^n, \Hom_R(\twisted{R}^n,E)\right)\cong
\Hom_R\left(E^n, \oplus_{\gamma\in\Gamma} \twisted{E}^n\right). $$
Also note that under this isomorphism
$B^\vee\in \Hom_R\left(E, \oplus_{\gamma\in\Gamma} \twisted{E}\right)^{n\times n}$ is given by
$(B^t_\gamma)_{\gamma\in\Gamma}$.
and that the image of $B^\vee$ is contained in $\oplus_{\gamma\in\Gamma} \twisted{E}^n$

Using the presentation
$ \twisted{R}^m \xrightarrow[]{\twisted{A}} \twisted{R}^n \rightarrow  \twisted{R}^n/ \twisted{V} \rightarrow 0$
we obtain the exact sequence
$$0 \rightarrow \left(\twisted{R}^n/\twisted{V}\right)^\vee \rightarrow
\Hom_R\left( \twisted{R}^n, E \right)  \xrightarrow[]{\twisted{A}^t} \Hom_R\left( \twisted{R}^m, E \right) $$
thus
$$ \left(\twisted{R}^n/\twisted{V}\right)^\vee = \Ann_{\Hom(\twisted{R}^n,E)} \twisted{A}^t .$$

We now obtain the commutative diagram
\begin{equation*}\label{CD16}
\xymatrix{
0 \ar@{>}[r]^{} &
\Ann_{E^n} A^t  \ar@{>}[r]^{} \ar@{>}[d]^{(B^t_\gamma T)_{\gamma\in \Gamma}} &
E^n  \ar@{>}[d]^{(B^t_\gamma T)_{\gamma\in \Gamma} } \\
0 \ar@{>}[r]^{}  &
\oplus_{\gamma\in\Gamma} \Ann_{\twisted{E}^n} \twisted{A}^t \ar@{>}[r]^{}  &
\oplus_{\gamma\in\Gamma} \twisted{E}^n \\
}
\end{equation*}
and we deduce that $\Ann_{E^n} A^t$ is a $R[\Theta_\gamma; f]$-module for all $\gamma\in \Gamma$.
\end{proof}

\begin{thm}\label{Theorem: finitely many submodules}
Let $M$ be an $R[\Theta; f]$-module with no nilpotents and assume $M$ is an Artinian $R$-module.
Then $M$ has finitely many $R[\Theta; f]$-submodules.
(Cf.~Corollary 4.18 in \cite{Blickle-Bockle}.)
\end{thm}
\begin{proof}
Write $\mathcal{M}=\mathcal{H}(M)$.
In view of \cite[Theorem 4.2]{Lyubeznik}, there is an injection between the set
of inclusions of $R[\Theta; f]$-submodules $N\subseteq M$ and the set of surjections of $F$-finite $F$-modules
$\mathcal{M} \rightarrow \mathcal{N}$ hence it is enough to show that there are finitely many such surjections.
By \cite[Theorem 2.8]{Lyubeznik} the kernels of these surjections are  $F$-finite $F$-submodules of $\mathcal{M}$ hence it is enough to show that
$\mathcal{M}$  has finitely many submodules.

All objects in the category of $F$-finite $F$-modules have finite length (cf.~\cite[Theorem 3.2]{Lyubeznik})
and the theorem now follows from \cite[Corollary 5.2(b)]{Hochster}.
\end{proof}

\begin{cor}\label{Corollary: finitely many submodules}
Let $B\in \Hom_R(\twisted{R}^n,R)$ be represented by
$(B_\gamma^t T)_{\gamma\in \Gamma} \in \mathcal{B}^{n\times n}$,
and assume that
$B_\gamma^t T: E^n \rightarrow  E^n$
is injective for some $\gamma\in \Gamma$.
Then there are finitely many $B$-compatible submodules of $\twisted{R}^n$.
In particular this holds when $n=1$ and $(B_\gamma T)_{\gamma\in \Gamma}: E \rightarrow  \oplus_{\gamma\in\Gamma} E$ is injective.
\end{cor}
\begin{proof}
Let $V$ be an $R$-submodule of $R^n$ and fix a matrix $A$ whose columns generate $V$.
Theorem \ref{Theorem: Frobenius and compatability 1} implies that if $\twisted{V} \subseteq \twisted{R}^n$ is
$B$-compatible then
for all $\gamma\in \Gamma$,
$\Ann_{E^n} A^t \subseteq E^n$ is an $R[\Theta; f]$-submodule of $E^n$ with the Frobenius action given by
$B_\gamma^t T$.
If there exists a $\gamma\in \Gamma$
such that $B_\gamma^t T$ is injective, then \cite[Theorem 3.10]{Sharp} or \cite[Theorem 3.6]{Enescu-Hochster} imply that there must
finitely many $R[B_\gamma^t T; f]$-submodules of $E^n$ and hence also finitely many $B$-compatible submodules of $R^n$.

Assume now that $n=1$.
For all $\gamma\in\Gamma$ write $Z_\gamma=\{ v\in E \,|\, B_\gamma T v=0 \}$
and let $C_\gamma\subseteq R$ be the ideal for which
$Z_\gamma=\Ann_{E} C_\gamma$.
If $C_\gamma\subseteq m R$ for all $\gamma\in \Gamma$, then $C=\sum_{\gamma\in\Gamma}  C_\gamma \neq R$,
and for any non-zero $v\in \Ann_{E} C\neq 0$, we have $B_\gamma T v=0$ for all $\gamma\in\Gamma$.
We conclude that there exists a $\gamma\in\Gamma$ such that, $C_\gamma=R$, i.e.,
that the Frobenius map $B_\gamma T$ on $E$ is injective, and the last assertion of the corollary follows.
\end{proof}

\section*{Acknowledgements}
I thank Karl Schwede for our pleasant discussions on Frobenius splittings and in
particular for showing me a variant of results in section \ref{Section: Applications to Frobenius splittings}
in the $F$-finite case.
I thank the anonymous referee for useful suggestions which improved the original version of this paper.

\end{document}